\RequirePackage[warning,log]{snapshot}
\documentclass{amsart}
\usepackage{graphicx}
\usepackage{amsmath}

\newtheorem{theorem}{Theorem}[section]
\newtheorem{lemma}[theorem]{Lemma}

\theoremstyle{definition}

\newtheorem{cor}[theorem]{Corollary}

\theoremstyle{remark}

\numberwithin{equation}{section}

\newcommand{\RR}{{\mathbb R}}
\newcommand{\ZZ}{{\mathbb Z}}

\begin{document}
\title{Inverting the wedge map and Gauss composition}

\author{Kok Seng Chua}
\email{chuakkss52@outlook.com}



\begin{abstract}
Let $1 \le k \le n,$ and let $v_1,\ldots,v_k$ be integral vectors in
$\ZZ^n$. We consider the wedge map $\alpha_{n,k} : (\ZZ^n)^k
/SL_k(\ZZ) \rightarrow \wedge^k(\ZZ^n)$, $(v_1,\ldots,v_k)
\rightarrow v_1 \wedge \cdots \wedge v_k $. In his Disquisitiones,
Gauss proved that $\alpha_{n,2}$ is injective when restricted to a
primitive system of vectors when defining his composition law for
binary quadratic forms. He also gave an algorithm for inverting
$\alpha_{3,2}$ in a different context on the representation of
integers by ternary quadratic forms. We give here an explicit
algorithm for inverting $\alpha_{n,2}$, and observe via Bhargava's
composition law for $\ZZ^2 \otimes \ZZ^2 \otimes \ZZ^2 $ cube that
inverting $\alpha_{4,2}$ is the main algorithmic step in Gauss's
composition law for binary quadratic forms. This places Gauss's
composition as a special case of the geometric problem of inverting
a wedge map which may be of independent interests. We also show that
a given symmetric positive definite matrix $A$ induces a natural
metric on the integral Grassmannian $G_{n,k}(\ZZ)$ so that the map
$X \rightarrow X^TAX$ becomes norm preserving.
\end{abstract}

\maketitle

\section{Introduction}
Let $1 \le k \le n$ and $ v_1,\ldots, v_k \in \ZZ^n$, consider the
$k$-wedge product map
$$X=(v_1, \cdots ,v_k) \rightarrow Y= v_1 \wedge \cdots \wedge v_k \in
\wedge^k \ZZ^n \cong \ZZ^{n_k},\; n_k=\begin{pmatrix} n \cr k
\end{pmatrix}$$ where for the index $I=i_1 \cdots i_k, \;\; 1 \le i_1 <i_2 <
\cdots < i_k \le n,$  the $I$th Pl\"{u}cker coordinate of $Y$ is
given by
\begin{equation} Y_I= \det \begin{pmatrix} (v_1)_{i_1} & \cdots
(v_k)_{i_1} \cr ... & ... \cr
 (v_1)_{i_k} & \cdots (v_k)_{i_k}  \end{pmatrix}.
 \end{equation}

 Since for any $k \times k$
matrix $A=(a_{ij})$, we have
 \begin{equation} (\sum_{j} a_{1j}v_j)
\wedge \cdots \wedge (\sum_j a_{kj} v_j) =(\det A) \;\; v_1 \wedge
\cdots  \wedge v_k, \end{equation}
 the wedge map descends to a
map of the orbits of $(\ZZ^n)^k \cong \ZZ^{nk}$ (identifying $X$
with a $n \times k$ integral matrix) under right multiplication by
$SL_k(\ZZ)$, and writing $G_{n,k}(\ZZ):=Im(\alpha_{n,k})$,
\begin{equation}
  \alpha_{n,k}: (\ZZ^n)^k/SL_k(\ZZ) \rightarrow G_{n,k}(\ZZ) \subset \wedge^k\ZZ^n.
\end{equation}
In general $\alpha_{n,k}$ is not one-one since for example, we have
$$(1,2,2)^T \wedge (1,0,4)^T=(0,1,-1)^T \wedge (2,0,8)^T =(8,2,-2)^T,$$
and the two pairs of bi-vectors  are not $SL_2$ equivalent, but note
that the vector $(8,2,-2)^T$ is not primitive. In general we will
say that a set of $k$ vectors $v_1,\cdots , v_k$ in $\ZZ^n$ forms a
primitive system if the gcd of the $n \choose k$ components of $v_1
\wedge \cdots \wedge v_k$ is one.

In his Disquisitiones (\cite{G} Art. 234), where he defined his
 composition law for binary quadratic forms, Gauss proved that the restriction of
 $ \alpha_{n,k} $ to  primitive systems of  vectors, which we denote by $\alpha_{n,k}^*$ is injective in
 the case $k=2$. This actually holds for all $k$ and we prove the
 case $k=3$ in Lemma 3.4 below which constructs an explicit $SL_3$
 transformation relating any two pre-images in $(\ZZ^n)^3$. The
 construction generalizes obviously for arbitrary $k$. Note that
 geometrically this just says that the "Pl\"{u}cker embedding" is
 indeed an embedding over Spec $\ZZ$.

 In his work on the representation of integers by ternary quadratic forms (\cite{G} Art.279), Gauss gave an
 algorithm for inverting $\alpha_{3,2}$. This is used in the
 construction of a
 correspondence between the representation of a binary quadratic form by the
 adjoint of a ternary quadratic form and the representation of the determinant of
 the binary form by the ternary form. This led him to his famous result
 expressing the number of representations of an integer as a sum of
 3 squares in term of class numbers. Inverting
 $\alpha_{n,n-1}$ is algorithmically the computation of  the kernel of a
 rank one integral matrix which may be effected by the Hermite normal form algorithm (see section 2.2).


In general $G_{n,k}(\ZZ)=Im(\alpha_{n,k})$ is only a proper subset
of $\wedge^k\ZZ^n$ when $1<k<n-1$, but a simple parameter count (see
comment after corollary 3.5) shows that $G_{n,k}(\ZZ)$ lies inside a
subvariety of codimension
 $d_{n,k}:=\begin{pmatrix} n \cr k
\end{pmatrix}-(kn-(k^2-1))$ of $\wedge^k\RR^n$. This is just the
codimension of the Grassmannian in its Pl\"{u}cker embedding.
Clearly, $d_{n,k}=d_{n,n-k}=0$ for $k=1,n-1$ and $d_{n,k}=1$  occurs
only once for $n=4,k=2$. This is perhaps the most interesting case
since the single relation allows one to impose a group law, and this
is exactly Gauss's composition law for binary quadratic forms or
more generally Bhargava's \cite{B} composition law for $\ZZ^2
\otimes \ZZ^2 \otimes \ZZ^2$ cubes. Indeed the
 invariance of the discriminants of the binary forms $Q_i^A,i=1,2,3$ attached to a Bhargava's
  $ 2\times 2\times 2$ cube $A$ is
 equivalent to $Im(\alpha_{4,2})$ lying on the  Pfaffian
 subvariety
 \begin{equation} X_{12}X_{34}-X_{13}X_{24}+X_{14}X_{23}=0,
 \end{equation}
 of $\ZZ^6$.

So the algorithmic  problem of
 given any two binary quadratic forms $Q_i(x,y),i \in \{2,3\}$ of the same
discriminant, to find a Bhargava's cube $A$ so that $Q_i^A=Q_i$ (see
 section 3.2) is equivalent to inverting $\alpha_{4,2}$. Both D. Shanks
 (see \cite{L}  F. Lemmermeyer )  and  \cite{P} Alf van der Poorten have considered
this algorithmic problem and \cite{P} also observed that it is
equivalent to inverting $\alpha_{4,2}$. In section 3.1 (Theorem
3.2), we give an explicit
 algorithm for inverting $\alpha_{n,2}$ which involves only a
 single GCD calculation and for general $n$, and this appears to be new.
By specializing, it gives rise to an explicit
 composition law for $\ZZ^2 \otimes \ZZ^2 \otimes \ZZ^2 $ cube and in particular recovers an old explicit rule for composing binary quadratic forms given
 by Arndt (see Theorem 3.3).

As Bhargava \cite{B} showed, the above  gives an algorithm for
composing $ 2\times 2\times 2 $ cubes (Theorem 2 \cite{B}) which
additionally gives rise to  algorithms for four new composition
laws.

In section 4, we show that Gauss's argument on inverting
$\alpha_{3,2}$ may be generalized so that any given symmetric
positive definite matrix $A$ induces a natural metric on
$G_{n,k}(\ZZ)$ so that the map $X \rightarrow X^TAX$ becomes norm
preserving, in the sense that an $X$ of $\hat{A}$ norm $m$ is mapped
to a $k-ary$ form of determinant $m$.

\section{Duality and Inverting $\alpha_{n,n-1}$}

\subsection{Notations and Duality}
 We will write $\wedge^k \ZZ^n$ for the products of
$\begin{pmatrix} n \cr k \end{pmatrix}$ copies of $\ZZ$ indexed by
$I=i_1,\ldots,i_k$ with $1 \le i_1 < i_2 < \cdots < i_k \le n$ and
we let $|I|=k$, $||I||=i_1+\cdots+i_k$. We also write
$G_{n,k}(\ZZ):=Im(\alpha_{n,k}) \subset \wedge^k\ZZ^n$ for
 the subset of vectors expressible as $Y=v_1 \wedge \cdots \wedge v_k$.

There is a map
$$\wedge^k\ZZ^n \rightarrow
\wedge^{n-k}\ZZ^n : Y \rightarrow \hat{Y},\;\;\;
\hat{Y}_I=(-1)^{||\hat{I}||}Y_{\hat{I}},$$
 where for any index $I$,
$\hat{I}=J=j_1\ldots j_{n-k}$ is the dual index with $I \cup
J=\{1,\ldots,n\}$ and $j_i$ increasing. Note that $\hat{\hat{I}}=I$
but $\hat{\hat{Y}}=(-1)^{(n(n+1))/2}Y$. We also write $gcd(Y)$ for
$gcd\{ Y_I: |I|=k\}$. We will always identify an $X=(v_1,\ldots,v_k)
\in (\ZZ^n)^k$ with the $n \times k$ matrix $X=[v_1 \cdots v_k]$ and
we will write $X_I$ for its $I$-th Pl\"{u}cker coordinate as given
in (1.1). We will say that $X$ is primitive system  if
$gcd_I(X_I)=1$ and we note that this is equivalent to
$\alpha_{n,k}(X)$ is a primitive vector in $\wedge^k\ZZ^n$ in the
usual sense. Note that a necessary and sufficient condition for
primitivity of $X$ is that it should be
 extendable to a $GL_n(\ZZ)$ matrix (see \cite{S}).

By the Cauchy-Binet identity \cite{CB}, we have a semidefinite
pairing on $(\ZZ^n)^k$ given by
\begin{equation}
<X,Y>:=\det (X^TY)=\sum_{|I|=k} X_I Y_I.
\end{equation}
The semidefiniteness follows from a Pythagoras theorem :
$$||X||^2 := \det(X^TX)=\sum_{|I|=k} (X_I)^2,$$
which  says that the square of the volume of the parallelepiped
 spanned by $v_1,\ldots,v_k$ in $\RR^n$ is the sum
of squares of the volume of its projections on the coordinate $k$-
planes. Moreover we have
$$<X,\hat{Y}>=<\hat{X},Y>$$ for $X \in (\ZZ^n)^k,Y \in
(\ZZ^n)^{n-k}$, though there is no linearity so that $<.,.>$ is
only a  semi-metric.

If in (2.1), $\alpha_{n,k}(X) =v_1 \wedge \cdots \wedge v_k \in
G_{n,k}(\ZZ)$ and  $ \alpha_{n,n-k}(Y)=w_1 \wedge w_2 \cdots \wedge
w_{n-k} \in G_{n,k}(\ZZ),$ are decomposable vectors of complementary
dimensions then
\begin{equation}
<X,\hat{Y}>=\det([v_1,\ldots,v_k,w_1,\ldots,w_{n-k}])=X \wedge Y,
\end{equation} by
the generalized column expansion of determinant.

\subsection{Inverting $\alpha_{n,n-1}$}
We will consider algorithms for inverting $\alpha_{n,k}$, by which
we mean finding any one of its pre-images. Note that by Lemma 3.4 we
can get all the images via an $SL_k$ transformation once we have
found one.

 We consider first the co-dimension one case.
Let $v_1,\dots,v_{n-1} \in \ZZ^n$, by (2.2),
$$<v_1 \wedge \cdots \wedge v_{n-1}, \hat{v_j}>=\det([v_1,\cdots v_{n-1},v_j])=0,$$
 for $1 \le j \le n-1$. This says $v_1 \wedge \cdots \wedge v_{n-1}$ is orthogonal to
 all the $v_j$ . So given $ 0 \neq v \in \wedge^{n-1}\ZZ^n \cong
\ZZ^n, gcd(v)=1$, the algorithm for inverting $\alpha_{n,n-1}$ is to
first find an integral basis $\{v_1,\ldots,v_{n-1}\}$ for the
rank-one system $v^Tx=0$ and we will have $\alpha_{n,n-1}(v_1 \wedge
\cdots \wedge v_{n-1})=v$. If $gcd(v)=\lambda$, one simply replaces
$v_1$ by $\lambda v_1$. So the problem reduces to computing the
kernel of a single row integral matrix.

Computing a basis of the kernel of an integral matrix $A$ can be
effected by the well known algorithm for computing  Hermite normal
form. We transform $A$ to HNF : $H=AU$ where $U \in GL_n(\ZZ)$,
then the columns of $U$ corresponding to zero columns of $H$ gives
a basis (see for example, \cite{C} Proposition 2.4.9). In PARI,
all information can be read off by a single command
$a=mathnf(A,2)$.

The above algorithm for inverting $\alpha_{n,n-1}$ should be
viewed as an algorithm for inverting $\alpha_{n,n-k}$ given that
one can invert $\alpha_{n,k}$ as given by the next lemma .

\begin{lemma} Let $0< k < n-k$ and  given $w_1,\cdots,w_k \in \ZZ^n$
such that $W=w_1 \wedge \cdots \wedge w_{k} \in \wedge^k(\ZZ^n)$
is not the zero vector. Then one can find vectors $v_1 ,\cdots,
v_{n-k} \in \ZZ^n$ with $v_1 \wedge \cdots \wedge v_{n-k}=
\hat{W}$.
\end{lemma}
\begin{proof} We can clearly assume $gcd(W)=1$ as before. We
compute a basis $v_1,\ldots,v_{n-k}$ for the kernel of the $k
\times n$ matrix with the $w_i$ as rows by a HNF algorithm so that
 $w_i \cdot v_j=0,i=1,\ldots,k,j=1,\ldots,n-k,\; gcd(v_1 \wedge \cdots \wedge v_{n-k})=1$. By (2.2),
 both $\hat{W}$ and $V=v_1 \wedge ... \wedge v_{n-k}$ are orthogonal to the adjoint of
$$v_{i_1} \wedge \cdots \wedge v_{i_k}, \;
(v_{i_1} \wedge \cdots \wedge v_{i_{k-1}}) \wedge
w_{j_1},\;(v_{i_1} \wedge \cdots \wedge v_{i_{k-2}})\wedge
(w_{j_1} \wedge w_{j_2}), \cdots, v_{i_1} \wedge (w_{j_1} \wedge
\cdots \wedge w_{j_{k-1}})$$ which spans a space of codimension
$\begin{pmatrix} n \cr n-k
\end{pmatrix}-
\sum_{j=0}^{k-1}\begin{pmatrix} n-k \cr k-j \end{pmatrix}
\begin{pmatrix} k \cr j \end{pmatrix} =1$ in $\wedge^k\RR^n$
so that $\hat{W}= \pm v_1 \wedge \cdots \wedge v_{n-k}$ since both
have components with gcd 1.
\end{proof}
Since inverting $\alpha_{n,1}$ is trivial, we get the algorithm
for inverting $\alpha_{n,n-1}$ above. For $k>1$ one needs to first
determine which points in $\wedge^k\ZZ^n$ are decomposable and a
separate algorithm is needed. We will do this first for $k=2$
which tells us how to generalize to arbitrary $k$.

\section{Inverting $\alpha_{n,2}$ and constructing Bhargava's
cube}

\subsection{Inverting $\alpha_{n,2}$} If $x=(x_1,x_2,x_3,x_4)^T, y=(y_1,y_2,y_3,y_4)^T \in \ZZ^4$ and
$X=x \wedge  y \in \ZZ^6$, a simple calculation shows that
$$X_{12}X_{34}+X_{14}X_{23}=X_{13}X_{24}.$$
It follows that $G_{n,2}(\ZZ)$ must lie in the intersections given
by the Pl\"{u}cker relations
\begin{equation}
X_{ij}X_{kl}-X_{ik}X_{jl}+X_{il}X_{jk}=0, \;\;\; 1 \le i < j < k<l
\le n,
\end{equation}
 of $\wedge^2\ZZ^n$. There are $\begin{pmatrix} n \cr 4
 \end{pmatrix}$ relations in (3.1), but they are clearly not independent. Given any non-zero $X$ (say with $X_{12} \neq
 0)$, we will show that (theorem 3.2 below) we can find $x, y \in \ZZ^n$ with $x \wedge y=X$ using only those relations where
 $X_{12}$ occurs :
\begin{equation}
X_{12}X_{kl}-X_{1k}X_{2l}+X_{1l}X_{2k}=0, \;\;\; 3 \le k<l \le n.
\end{equation}

  We note that this gives exactly $\begin{pmatrix} n-2 \cr 2
 \end{pmatrix}$ independent relations, which agree with the codimension $d_{n,2}$.

 We will need first a Lemma from D. Cox
 (\cite{Cox}, Lemma 3.5) who may have anticipated the geometric view as he stated his result
 in the exact determinantal form that we need.
 \begin{lemma} Let $p=(p_1, \ldots , p_r)^T, q=(q_1, \ldots,
 q_r)^T
 \in \ZZ^r$ and $y$ be an integer such that $gcd(y,p_1,
 \ldots, p_r)=1$. Then the congruence $p x \equiv q$ mod $y$ has a
 unique solution mod $y$ if and only if $p \wedge q=0$ mod $y$.
\end{lemma}
We note that moreover $x$ in the above is easily computable by a
single GCD computation. Let $L \in \ZZ^r, \ell \in \ZZ$ be such that
$L \cdot p+ \ell y=1$. Then $L \cdot p \equiv 1 $ mod $y$, so we
have
 \begin{equation}
 x \equiv L \cdot q\;\;  mod \;\; y.
 \end{equation}

\begin{theorem}
The  Pl\"{u}cker relations (3.1) defined the integral Grassmannian
 $G_{n,2}(\ZZ)=Im(\alpha_{n,2})$  in the Pl\"{u}cker embedding
 in $\wedge^2\ZZ^n$. There is an explicit algorithm
(involving only a single GCD computation) for finding a pre-image of
any $0 \neq X \in G_{n,2}(\ZZ)$ using only $d_{n,2}$ of the
relations
 containing a particular $X_{ij}$ where $X_{ij} \neq 0$.
\end{theorem}

\begin{proof} Let $X \in Im(\alpha_{n,2})$, the relations (3.1) are
clearly necessary. Without loss of generality we may assume that
 $X_{12} \neq 0$, and we will find $x, y \in \ZZ^n$ with $x \wedge
 y =X$. By the $SL_2(\ZZ)$ invariance, if we have a solution, we can bring the first two
 rows of $[x,y]$ into Hermite normal form, so we can expect to
 find $x,y$ with $y_1=0, x_1 y_2= X_{12}$ and $|x_2|<|y_2|$ which
 leads us to the following algorithm.

 We first compute integers $\lambda_j$ (and this is the only computation
 needed) such that
 \begin{equation}
 x_1 = gcd(X_{12},\ldots,X_{1n}) =\sum_{j=2}^n \lambda_j X_{1j},
 \end{equation}
 and set
 $$
 y_1 = 0,\;\;\; y_j=\frac{X_{1j}}{x_1}, j=2, \ldots, n.
 $$
  We then let $p=(y_3,\ldots
 y_n)^T, q=(X_{23},\ldots, X_{2n}) \in \ZZ^{n-2}$. Since
  $gcd(y_2,y_3,\ldots, y_n)=gcd(X_{12}/x_1,\ldots, X_{1n}/x_1)=1$, and  by (3.2) for $3 \le k<l\le n$,
 $$p_kq_l-q_kp_l=\frac{X_{1k}X_{2l}-X_{1l}X_{2k}}{x_1}=\frac{-X_{12}X_{kl}}{x_1}=-y_2X_{kl},$$
 so we have $p \wedge q \equiv 0$ mod $y_2$.
By Lemma
 3.1 there is a unique $x_2$ mod $y_2$ such that
 $$ x_2 y_j \equiv X_{2j} \;\; mod \; \; y_2, j=3, \ldots n,$$
which by (3.3) and the above, is given explicitly by
\begin{equation}
  x_2=\sum_{j=3}^n \lambda_j X_{2j} \;\;mod \;\; y_2,
  \end{equation}
  where the $\lambda_j$ is computed earlier in (3.4)
so that
 \begin{equation}
 x_j:=\frac{x_2y_j-X_{2j}}{y_2} \in \ZZ, \;\; j=3,\ldots,n.
 \end{equation}
We have thus found vectors $x, y \in \ZZ^n$ with
\begin{equation}
X_{ij}=x_iy_j-x_jy_i, 1 \le i  <j \le n, i \le 2.
\end{equation}

 It then follows from (3.2) and  (3.7) that for $2< i <j  \le n$, we have

\begin{eqnarray*}
   X_{ij}&=&\frac{ X_{1i}X_{2j}-X_{1j}X_{2i}}{X_{12}}\cr &=&
\frac{(x_1y_i-y_1x_i)(x_2y_j-y_2x_j)-(x_1y_j-y_1x_j)(x_2y_i-y_2x_i))}{x_1y_2-y_1x_2}
=x_iy_j-y_ix_j.
 \end{eqnarray*}
 So we have found $x,y$ with $X=x \wedge y$ using only (3.2), but
 (3.1) must now holds since each $X_{ij}=x_iy_j-x_jy_i$.
\end{proof}

\subsection{Constructing Bhargava's cube by inverting
$\alpha_{4,2}$} We now show via Bhargava's \cite{B} composition law
on $ 2\times 2\times 2$ cubes that inverting $\alpha_{4,2}$ is the
main algorithmic step in Gauss's composition law for binary
quadratic forms. This seems to give an interesting geometric and
simple way to view this important algorithm.

Let $A$ be a Bhargava's \cite{B} $2 \times 2 \times 2 $ cube of
integers, which is an optimally symmetric description. We will fix
one of the three (say the front/back) slicing and let
\begin{equation}
M_1=\begin{pmatrix} x_1 & x_2 \cr x_3 & x_4 \end{pmatrix},\;\;
   N_1=\begin{pmatrix} y_1 & y_2 \cr y_3 & y_4 \end{pmatrix},
\end{equation}
   be the front and back faces. We also have
\begin{eqnarray}
M_2=\begin{pmatrix} x_1 & x_3 \cr y_1 & y_3
\end{pmatrix},\; \;&&
   N_2=\begin{pmatrix} x_2 & x_4 \cr y_2 & y_4 \end{pmatrix}, \cr
M_3=\begin{pmatrix} x_1 & y_1 \cr x_2 & y_2 \end{pmatrix},\;\;&&
   N_3=\begin{pmatrix} x_3 & y_3 \cr x_4 & y_4 \end{pmatrix},
 \end{eqnarray}
   corresponding to the left/right, top/bottom faces. The
   associated binary quadratic forms are
   $$ Q_i^A(x,y)= -\det(M_ix-N_iy),\;\; i=1,2,3,$$
   and these sum to zero in the class group.

Corresponding to our slicing are two vectors
$x=(x_1,x_2,x_3,x_4)^T, y=(y_1,y_2,y_3,y_4)^T \in \ZZ^4$ and their
wedge product $X= x \wedge y$ with components $X_{ij}=\det
\begin{pmatrix}  x_i & y_i \cr x_j & y_j \end{pmatrix}$ satisfying (1.4). A
calculation gives
\begin{equation} Q_1^A=(-\det M_1, \det
\begin{pmatrix} x_1 & x_2 \cr y_3 & y_4
\end{pmatrix}- \det \begin{pmatrix}  x_3 & x_4 \cr y_1 & y_2
\end{pmatrix}, -\det N_1),
\end{equation}
\begin{eqnarray}
Q_2^A &=& (-X_{13},X_{14}+X_{23},-X_{24}) , \cr
 Q_3^A &= &(-X_{12},
X_{14}-X_{23}, -X_{34}),
\end{eqnarray}
 and we note  the
invariance of $Disc(Q_i^A)$ is equivalent to (1.4).

Suppose now we are given two primitive binary quadratic forms
$Q_2=(a_2,b_2,c_2),Q_3=(a_3,b_3,c_3)$ of the same discriminant, then
the vector
\begin{equation}
 X=\begin{pmatrix} X_{12} \cr
X_{13} \cr X_{14} \cr X_{23}\cr X_{24} \cr X_{34} \end{pmatrix} :=
\begin{pmatrix} -a_3 \cr -a_2 \cr (b_2+b_3)/2 \cr (b_2-b_3)/2 \cr
-c_2 \cr -c_3
\end{pmatrix} \in Im(\alpha_{4,2}),
\end{equation}
since $X$ satisfies (1.4) by the equality of the discriminants.
 Furthermore by the primitivity of $Q_2,Q_3$, $X$ is primitive
($gcd(X)=1$).
 By our algorithm for inverting $\alpha_{4,2}$ , we can find $x, y
 \in \ZZ^4$ with $x \wedge y=X$. We can now construct a $2 \times 2 \times 2 $ cube $A$
 with $x,y$ as the front/back face and we get $Q_i^A=Q_i, i=2,3$.
 so we have found $[Q_1^A]=-([Q_2]+[Q_3])$ in the class group as Bhargava had shown.
 Also by section 3.3 below $\alpha_{4,2}^*$ is injective, so  $x,y$ and hence $Q_1^A$ is determined up
 to $SL_2(\ZZ)$ transformation, so that there is a unique composed
 class. By applying the explicit algorithm for inverting $\alpha_{4,2}$ in Theorem 3.2 to the vector
 $X$, we obtain the following simple algorithm for constructing Bhargava's
 cube which leads to an explicit algorithm for Gauss's composition
 agreeing with that given in (\cite{Bu},Theorem 4.10) attributed to Arndt.

 \begin{theorem} Given two primitive binary quadratic forms
 $Q_2=(a_2,b_2,c_2), Q_3=(a_3,b_3,c_3)$ of the same discriminant $D$. Let
 \begin{eqnarray}
 x_1&=&gcd(-a_3,-a_2,(b_2+b_3)/2) \cr &=&
  \lambda_1(-a_3)+\lambda_2(-a_2)+\lambda_3
 (b_2+b_3)/2,\;\;\; \lambda_i \in \ZZ \cr
 x_2 &=&\lambda_2(b_2-b_3)/2-\lambda_3 c_2,
 \end{eqnarray}
then under Gauss' composition, a representative of $[Q_2 + Q_3]$
is given
 by the primitive form
 \begin{equation}
  Q_1'=(c_1,b_1,a_1):=\left(\frac{a_2a_3}{x_1^2}, b_2+2a_2 \frac{x_2}{x_1},
 \frac{(b_1^2-D)x_1^2}{4a_2a_3} \right),
 \end{equation}
 which specializes  to Dirichlet composition (\cite{Cox} Proposition
 3.8) when $x_1=1.$ More over if we set
 \begin{eqnarray}
   x_3 &=& (2a_2x_2+(b_2-b_3)x_1)/(2a_3) \cr
   x_4 &=& (2c_2x_1 +(b_2+b_3)x_2)/(-2a_3) \cr
   y_1 &=& 0 \cr
   y_2 &=& -a_3/x_1 \cr
   y_3 &=& -a_2/x_1 \cr
   y_4 &=& (b_2+b_3)/(2x_1),
   \end{eqnarray}
   and let $A$ be the Bhargava's cube with front and back faces given
   by (3.8). Then $Q_i^A=Q_i, i=2,3$.
The algorithm involves  only a single GCD computation.
\end{theorem}
\begin{proof} By applying the algorithm in Theorem
3.2 to $X \in Im(\alpha_{4,2})$ given by (3.12), we find $x,y$
with $x \wedge y =X$ from (3.4),(3.5),(3.6) which gives (3.13),
(3.15) and the cube $A$ with $Q_1=(a_1,b_1,c_1)$ from (3.10)
 where $a_1,b_1,c_1$ are as defined in (3.14). Since
 $[Q_1]=-[Q_2+Q_3]$, we get  $-Q_1=(a_1,-b_1,c_1)$ for the composed  class $[Q_2+Q_3]$,
 and  applying the $SL_2(\ZZ)$ transformation $x'=-y,y'=x$ to $-Q_1$ gives
 the more standard
equivalent form $Q_1'=(c_1,b_1,a_1)$ in (3.14). The only
algorithmic step is (3.13).
\end{proof}

\subsection{Further composition laws on cubes}
The construction above also solves the algorithmic problem of
finding the
 composition of two Bhargava's projective cubes (\cite{B} Theorem
 2). Given two projective cubes $A,A'$ we can compute by the above
primitive forms
 $Q_2,Q_3$ so that $$[Q_2]=[Q_2^A]+[Q_2^{A'}],
 [Q_3]=[Q_3^A]+[Q_3^{A'}],$$
 we can then find a cube $A''$ with $Q_i^{A''}=Q_i,i=2,3$ by the above.
 We then  have
\begin{eqnarray*}
 [Q_1^{A''}]&=&-( [Q_2^{A''}]+[Q_3^{A''}])=-([Q_2]+[Q_3]) \cr
            &=& -([Q_2^A]+[Q_2^{A'}]+[Q_3^A]+[Q_3^{A'}]) =
            ([Q_1^A]+[Q_1^{A'}]),
\end{eqnarray*}
so $[Q_i^{A''}]=[Q_i^A]+[Q_i^{A'}],i=1,2,3$, and hence
 $[A'']=[A]+[A']$ and $[A'']$ is projective since composition of
 primitive forms is primitive. As Bhargava showed in \cite{B}, the
 composition law on cubes gives rise to 4 further composition laws
 and hence the algorithm extends to these also.

 \subsection{$\alpha_{n,k}^*$ is injective} In (\cite{G} Art 234),
 Gauss proved that if $a,b,c,d$ are vectors in $\ZZ^n$ with
 $ c \wedge d = t (a \wedge b), t \in \ZZ$ and $gcd(a \wedge b)=1,$ then
 $[c,d]=[a,b]H, $ where $H \in GL_2(\ZZ), \det(H)=t$. In particular
 this implies that $\alpha_{n,2}^*$ restricted to primitive system
 of
 vectors is injective. Gauss' proof can be generalized and the
 matrix $H$ can be given explicitly. We give here the proof for the
 case $k=3$ and the explicit form of $H$ which makes it
 clear that the result holds for all $k$. We are indebted to the
 referees of ANTS VIII for the following which greatly
 simplifies our earlier proof.

 \begin{lemma} (case $k=3$)
Let $a,b,c,e,f,g$ be vectors in $\ZZ^n$ and $0 \neq t \in \ZZ$ with
$e \wedge f \wedge g = t (a \wedge b \wedge c),$ and $gcd(a \wedge b
\wedge c)=1$. Let $L \in \ZZ^n$ be such that $L \cdot (a\wedge b
\wedge c)=1$, and let $H$ be the following integral matrix (where
$\cdot$ denotes inner product)
$$ H=L \cdot \begin{pmatrix}
 e \wedge b \wedge c & f \wedge b \wedge c & g \wedge b \wedge c \cr
 a \wedge e \wedge c & a \wedge f \wedge c & a \wedge g \wedge c \cr
a \wedge b \wedge e & a \wedge b \wedge f & a \wedge b \wedge g
\end{pmatrix},$$
then we have
$$[a,b,c]H=[e,f,g], \;\;\;\; \det H=t .$$
\end{lemma}

\begin{proof}
We need to show that
\begin{equation} L \cdot (e
\wedge b \wedge c) a + L \cdot (a \wedge e \wedge c ) b + L \cdot
(a \wedge b \wedge e )c =e, \end{equation}
 for the first entry and
the rest follows similarly. Clearly, $a\wedge b \wedge c \wedge
e=0$, so we know $a,b,c,e$ are linearly dependent. But since we
assume that $gcd(a \wedge b \wedge c)=1$, $a, b, c$ are linearly
independent, so we have $e=\lambda_1 a+\lambda_2 b +\lambda_3 c$,
for $\lambda_i \in \RR$. It follows that $L \cdot (e \wedge b
\wedge c)=L \cdot ( \lambda_1 a \wedge b \wedge c)= \lambda_1 \in
\ZZ,$ and likewise for $\lambda_2,\lambda_3$.
\end{proof}

The lemma implies the following :
\begin{cor} The wedge map $\alpha_{n,k}$ gives a bijection
\begin{equation}
(\ZZ^n-0)^{k*}/SL_k(\ZZ) \cong (G_{n,k}(\ZZ))^*,
\end{equation}
where $*$ denotes restriction to nonzero primitive system of
vectors.
\end{cor}
(3.17) implies a simple counting argument which gives our formula
for $d_{n,k}$ in the introduction. We first note that restriction to
primitive system of vectors does not reduce dimension. Thinking of
$\alpha_{n,k}$ in (1.3) as a covering map mapping $\ZZ^{kn}$ onto
$G_{n,k}(\ZZ)$ with "fiber" $SL_k(\ZZ)$  depending on $k^2-1$
parameters gives us $kn-(k^2-1)$ free parameters for $G_{n,k}(\ZZ)$,
and hence the formula for $d_{n,k}$.

\section{Inverting $\alpha_{n,k}$ and representations of forms}
In (\cite{G} Art.279), Gauss gave an algorithm for inverting
$\alpha_{3,2}$. It is interesting to note that he was concerned here
with the representation of integers by ternary form, rather than his
composition law. We generalized his argument here which gives one
reason why  one may wish to invert $\alpha_{n,k}$.

Let $A$ be a fixed matrix for an $n$-ary form and $X \in (\ZZ^n)^k$,
then the well known map  $\phi_A: X \rightarrow X^TAX$ transforms
$A$ to a $k$-ary form. Two different $X$ define equivalent $k$-ary
form if and only if they differ by an $SL_k(\ZZ)$ transform, so one
should think of $X$ as lying in the integral Grassmannian
$G_{n,k}(\ZZ)$. Furthermore, as we shall show, $A$ induces a natural
metric on $G_{n,k}(\ZZ)$ so that $\phi_A$ becomes volume preserving.

 For $X=[v_1,...,v_k],Y=[w_1,...,w_k] \in (\ZZ^n)^k$, (2.1) gives a "flat" semi-definite pairing.
 The matrix $A$ induces
correspondingly a semi-definite pairing on $(\ZZ^n)^k$, given by
\begin{equation}<X,Y>_A:=\det(X^TAY)=\sum_{|I|,|J|=k} \hat{A}_{IJ} X_I
Y_J, \end{equation} where $\hat{A}$ is the $k$-skew symmetric
adjoint of $A$ with $IJ$ component given by

\begin{equation} \hat{A}_{IJ}=  \det \begin{pmatrix} A_{i_1 j_1} & \cdots
A_{i_1,j_k} \cr ... & ... \cr
 A_{i_k,j_1} & \cdots A_{i_k j_k}  \end{pmatrix}.
 \end{equation}
(4.1) follows from applying the Cauchy-Binet identity (2.1) twice,
and hence it is an algebraic identity which holds even when $A$ is
not necessarily symmetric or positive definite.

Putting $Y=X$, (4.1) says that the representation of a $k$-ary form
\begin{equation}\phi_k=X^TAX \end{equation} by an $n$-ary form $A$
induces a representation of the integer $\det \phi_k$ by the
$n_k$-ary adjoint form $\hat{A}$ of $A$,
\begin{equation}\det \phi_k =x^T \hat{A} x,
\end{equation}
by setting for  $X=[v_1,...,v_k]$,
\begin{equation}
 x=v_1 \wedge ... \wedge v_k .
 \end{equation}

 However passing conversely from (4.4) to (4.3) is only
possible when one can invert $\alpha_{n,k}$ at $x$
 in (4.5), which also requires
 $ x \in Im(\alpha_{n,k})=G_{n,k}(\ZZ)$. Theorem 3.2 gives an
 algorithm to do this when $k=2$ and Lemma 2.1 gives the algorithm
 for the codimension one case where $\alpha_{n,n-1}$ is always surjective.

It seems interesting to note that the above can be viewed
geometrically as giving a natural $\hat{A}$-norm on the integral
Grassmannian $G_{n,k}(\ZZ)$ so that $\phi_A$ maps vectors of norm
$m$ into volume $m$ $k$-dimensional sublattices of the integral
lattice $\Gamma_A$ with Gram matrix $A$.
\begin{lemma} Let $A$ be an integral symmetric positive definite $n \times n$ matrix and $\Gamma_A$
 the (classical) integral lattice with
  Gram matrix $A$. Let $\hat{A}$ be the $k$-adjoint of $A$  as given in
  (4.2) which induces a natural metric on the integral Grassmannian
  $G_{n,k}(\ZZ)$. Let $G_{n,k}(\ZZ)^*_{\hat{A},m}=\{ x : x^T\hat{A}x=m, gcd(x)=1 \}$ be
  the
  primitive vectors in $G_{n,k}(\ZZ)$ with $\hat{A}$-norm $||x||^2_{\hat{A}}:=x^T\hat{A}x=m$.
  Let $\Gamma^*_{A,k,m}=\{ [X^TAX]_{SL_k(\ZZ)} : \det(X^TAX)=m, gcd(X)=1\}$,be the
  isomorphism classes of primitive $k$ dimensional sublattices of
  $\Gamma_A$ of determinant m. Then there is a natural map
$$ \phi_{A,k,m} : G_{n,k}(\ZZ)^*_{\hat{A},m} \rightarrow \Gamma^*_{A,k,m}$$ mapping
$$ x=v_1 \wedge \cdots \wedge v_k \rightarrow [X^TAX].$$
\end{lemma}

We also note that there is an obvious injection
$$\Gamma^*_{A,k,m} \subset H_k(m)=\{ [\phi_k] : \det (\phi_k)=m \},$$
and it is onto if and only if every classical integral
$k$-dimensional primitive lattice of determinant $m$ can be embedded
as a sublattice of $\Gamma_A$. This is known to be true for example
for $A=I,k \le 5$ and $n \ge k+3$ \cite{Ko,M}, by Chao Ko and
Mordell, the case $k=1,n=4$ being Lagrange theorem that every
integer is a sum of 4 four squares.

\section{Acknowledgements}
This paper was first reviewed for ANTS VIII. We are very much
indebted to the referees of (ANTS VIII and IX) for pointing out an
error in our initial version of Theorem 3.2, and for simplifying the
proof of Lemma 3.4, as well as giving many useful comments and
suggestions, which greatly improve our presentations. Some of our
results were found while preparing a talk at the Max-Planck-Institut
f\"{u}r Mathematik in July 2008. We thank the institute for
hospitality.

\bibliographystyle{amsplain}

\end{document}